\documentclass[letterpaper, 10 pt, conference]{ieeeconf}  

\IEEEoverridecommandlockouts                              
\title{\LARGE \bf
A Quantization Procedure for Nonlinear Pricing with an Application to Electricity Markets
}

\author{Quentin Jacquet, Wim van Ackooij, Clémence Alasseur and Stéphane Gaubert
\thanks{Q. Jacquet, W. van Ackooij and C. Alasseur are with EDF R\&D Saclay, Palaiseau, France
        {\tt\small \{quentin.jacquet, wim.van-ackooij, clemence.alasseur\}@edf.fr}}%
\thanks{Q. Jacquet and S. Gaubert are with INRIA, CMAP, Ecole Polytechnique, IP Paris, CNRS, Palaiseau, France
        {\tt\small stephane.gaubert@inria.fr}}%
}
\let\labelindent\relax
\usepackage{enumitem}
\usepackage{dsfont,amsmath,amssymb,mathtools}

\DeclareMathOperator*{\argmin}{arg\,min}

\DeclareMathOperator{\bbR}{\mathbb{R}}

\DeclareMathOperator{\calC}{\mathcal{C}}

\DeclareMathOperator{\calE}{\mathcal{E}}
\DeclareMathOperator{\calK}{\mathcal{K}}

\DeclareMathOperator{\calJ}{\mathcal{J}}
\DeclareMathOperator{\calL}{\mathcal{L}}

\DeclareMathOperator{\calU}{\mathcal{U}}

\newcommand{\dx}{\textnormal{d}x}
\newcommand{\cz}{\check{z}}

\newcommand{\ce}{\check{x}}

\newcommand{\shortminus}{\scalebox{0.5}[1.0]{$-$}}

\usepackage{hyperref}
\usepackage[capitalise]{cleveref}

\newtheorem{prop}{Proposition}[section]
\newtheorem{theorem}[prop]{Theorem}

\newtheorem{hypo}{Assumption}[section]

\DeclareMathSymbol{\mlq}{\mathord}{operators}{``}
\DeclareMathSymbol{\mrq}{\mathord}{operators}{`'}

\makeatletter
\newcommand*\frob{\mathpalette\bigcdot@{.7}}
\newcommand*\bigcdot@[2]{\mathbin{\vcenter{\hbox{\scalebox{#2}{$\m@th#1\bullet$}}}}}

\usepackage[colorinlistoftodos,bordercolor=orange,backgroundcolor=orange!20,linecolor=orange,textsize=scriptsize]{todonotes}

\usepackage[font=small,justification=centering]{caption}
\usepackage[justification=centering]{subcaption}

\usepackage{textcomp}
\usepackage{algorithm}
\usepackage{algpseudocode}

\def\namedlabel#1#2{\begingroup
    #2%
    \def\@currentlabel{#2}%
    \phantomsection\label{#1}\endgroup
}

\newcounter{savealgorithm}
\newenvironment{subalgorithms}
 {%
  \stepcounter{algorithm}%
  \edef\currentthealgorithm{\thealgorithm}%
  \setcounter{savealgorithm}{\value{algorithm}}%
  \setcounter{algorithm}{0}%
  \renewcommand{\thealgorithm}{\currentthealgorithm\alph{algorithm}}%
 }
 {%
  \setcounter{algorithm}{\value{savealgorithm}}%
 }

\usepackage[sorting=none,natbib=true,backend=biber,doi=false,url=false,
isbn=false,eprint=true,giveninits=true,maxbibnames=99]{biblatex}
\begin{document}

\maketitle
\thispagestyle{empty}
\pagestyle{empty}

\begin{abstract}
  We consider a revenue maximization model, in which a company aims at designing a menu of contracts, given a population of customers. A standard approach consists in constructing an incentive-compatible continuum of contracts, i.e., a menu composed of an infinite number of contracts, where each contract is especially adapted to an infinitesimal customer, taking his type into account. Nonetheless, in many applications, the company is constrained to offering a limited number of contracts. We show that this question
reduces to an optimal quantization problem, similar to the {\em pruning problem} that appeared in the max-plus based numerical methods in optimal control. We develop a new quantization algorithm, which, given an initial menu of contracts, iteratively prunes the less important contracts, to construct an implementable menu of the desired cardinality, while minimizing the revenue loss. We apply this algorithm to solve a pricing problem with price-elastic demand, originating from the electricity retail market. Numerical results show an improved performance by comparison with earlier pruning algorithms.
\end{abstract}

\section{INTRODUCTION}

\subsection{Motivation from electricity markets}
Electricity retail markets are now open to competition in most countries, and providers are free to design a  \emph{menu} of offers/contracts in addition to regulated alternatives (fixed prices), so that each consumer can select among the vast jungle of offers the one which maximizes his utility. Finding an appropriate utility function which fairly represents the consumer behavior is all but immediate. In this paper, the choice of a contract is based on the minimization of the invoice (rational choice theory, see e.g.~\citep{Scott_2000}), and we suppose that each customer can adjust his consumption to the electricity prices (\emph{price elasticity}). 
This phenomenon is highlighted by the actual energy crisis: consumers are likely to make huge efforts in view of consumption reduction. 

A key problem for electricity providers is to design an optimal menu of offers,
maximizing their revenue, under a restriction on  the ``size'' of the menu (number of contracts).
In fact, from an optimization point of view, proposing more contracts increases the revenue, as it allows one to adjust the menu to the individual preferences of the different types of customers. However, in practice, it is essential to restrict the number of contracts, in order to make the commercial offer more visible to agents, easier to understand, and also to keep an implementable menu for the company.

\subsection{The optimal nonlinear pricing problem}
We consider more generally the revenue maximization problem faced by a seller, called \emph{principal},
or \emph{leader} in the setting of Stackelberg games~\cite{Simaan_1973}. This problem has been addressed by the theory of mechanism design~\cite{Borgers_2015} through the question of nonlinear pricing. The so-called \emph{monopolist problem} is among the most studied ones: in this approach, the population is represented as a continuum of buyers (called \emph{agents} or \emph{followers}), and a contract can be specifically designed for each agent (continuum menu). In the seminal paper~\citep{Rochet_1998}, Rochet and Chon\'e study the monopolist problem by introducing a dual approach. In some specific cases (linear-quadratic setting and specific agents distribution), analytic solutions can be found in one~\cite{Mussa_1978} or many dimensions~\cite{Armstrong_1996}, via reformulation as welfare maximization using virtual valuation technique. Extending the framework of Rochet and Choné to decomposable variational problem under convexity requirement, Carlier~\cite{Carlier_2017} addresses the question of the existence and uniqueness of a solution, and proposes an iterative algorithm.  In the specific case $\bbR^2$, Mirebeau~\cite{Mirebeau_2016} introduces a more efficient method using an adaptive mesh based on stencils. The infinite-size menu is therefore characterized by a value-function satisfying the incentive-compatibility conditions as with the full-participation condition, the latter supposing that contracting with the whole population is optimal. Bergemann, Yeh and Zhang
recently considered the question of the optimal quantization of a menu~\cite{Bergemann_2021}.

\subsection{Contributions}
Our main contribution is the development of new \emph{quantization} algorithms which, given the infinite-size menu, aim at finding the best $n$-contracts approximation that maximizes the revenue. This 2-step strategy bypasses the combinatorial difficulty tackled in bilevel pricing -- see e.g.~\cite{Labbe_1998,Baldwin_2019} -- where formulations directly embed customer choices over the $n$ contracts, becoming rapidly untractable for large size of menu. We show that the quantization problem is equivalent to the {\em pruning problem}, which arose, following McEneaney~\cite{McEneaney_2007}, in the development of the max-plus based curse-of-dimensionality attenuation methods in numerical optimal control, see~\cite{McEneaney_2008,Gaubert_2011,Gaubert_2014}, and~\cite{mceneaneydower} for an application. In these methods, the value function of an optimal control problem is represented as a supremum of ``basis functions'', and one looks for a sparse representation -- with a prescribed number of basis functions.  In the present application, the basis functions are linear functions, representing contracts. 
We develop a {\em greedy descent} algorithm which iteratively removes the less ``important'' contracts.
We consider different importance measures, taking into account the $L_1$ and $L_\infty$ approximation errors previously considered in the study of the pruning problem, and also a specific measure of the loss of revenue, see~\Cref{alg:pruning,algo::up1}. An essential feature of these algorithms is the low incremental cost per iteration, with an update rule requiring only {\em local} computations
-- in a ``small neighborhood'' of the active set of a basis function. To do so, we exploit 
discrete geometry techniques, by associating to a basis decomposition a polyhedral complex, which
is updated dynamically.

To apply this algorithm to the optimal design of a menu in the electricity retail market, we generalize the framework of~\cite{Carlier_2017} to allow for a nondecomposable (still convex) cost. Indeed, the revenue of the provider depends on the furniture cost, supposed to be an increasing function of the \emph{global} consumption, see e.g.~(\cite{Ackooij_2018, Alekseeva_2019}).  In this extended setting, we prove the existence and uniqueness of the solution, see~\Cref{theorem::existence_1}. The solving of this problem is then tackled by a direct method (discretization of the variational problem). A key feature of the pricing application is the elastic behavior of the customers, who adapt their consumption according to prices.
We show that, after an appropriate change of variables, this actually reduces to the previous model,
see~\Cref{theorem::mf_problem_2}.
Numerical tests, on a realistic instance (arising from the French electricity market),
illustrate the efficiency of our approach both in terms of revenue gain and of computational time, see~\Cref{fig::errors}. Our algorithm also allows one to estimate the minimal admissible number of contracts,
given a target level of acceptable revenue loss by comparison with the case of an infinite number
of contracts. 

\subsection{Related works}
In the nonlinear pricing context, the restriction to a finite number of offers has been regarded only recently. In \cite{Bergemann_2021}, the authors analyze the loss of revenue induced by this restriction, exhibiting upper bounds of order $1/n^{2/d}$, where $d$ is the dimension and $n$ the maximal number of contracts. A similar asymptotic error rate arose in a different setting of quantization theory, see e.g.~\cite{Gaubert_2011}. Moreover, in the linear-quadratic setting of~\cite{Bergemann_2021}, the extreme distributions realizing the worst revenue loss satisfies separability conditions à la Armstrong~\cite{Armstrong_1996}, leading to an explicit expression for the optimal quantization. We do not satisfy these requirements here, as we tackle a broader class of variational problem, hence the need of efficient methods to solve the pricing problems with a finite number of contracts. In~\cite{Ekeland_2010}, a discretization is obtained by writing the utility function as a supremum of finitely many affine functions, and so the solution they obtain can be viewed as a $n$-contracts menu. However, the scheme also discretizes the population (with the same size as the contracts). In the present application, this is not desirable, since the size of the population has to be much larger that the size of the menu.

The present algorithms should be compared with the pruning methods to compute
a sparse representation of a function as maximum of a prescribed number of basis functions.
The pruning problem was shown in~\cite{Gaubert_2011} to be a continuous version
of the facility location problem, a hard combinatorial optimization problem.
The pruning algorithms developed in~\cite{McEneaney_2008,Gaubert_2011} rely
on a notion of importance metric, measuring the contribution of each basis
function to the approximation error. A basic algorithm in~\cite{McEneaney_2008,Gaubert_2011}
perform a single pass which keeps only the $n$ basis functions with the highest importance metric, the latter being evaluated either by solving a convex programming
problem or in approximate way, after a discretization of the state space.
A {\em greedy ascent} algorithm is also implemented in~\cite{Gaubert_2011},
adding incrementally functions by decreasing order of importance.
In contrast, the present algorithm does not require a discretization
of the state space. Moreover, the use of fast (local) updates of the importance measure allows us to perform a greedy descent 
starting from the complete family of basis functions, and removing at each
stage the less important one. This leads to improved performances on
our application case. 

The paper is organized as follows: in~\Cref{sec::model}, we define the nonlinear pricing problem, adapted to our application case, and encompassing the monopolist framework. In~\Cref{sec::pruning}, we approximate the continuum menu by a finite set of contracts, and present refined pruning algorithms with local update. Then, in~\Cref{sec::app}, we specify the problem encountered in electricity markets, and show how it boils down to the general case of~\Cref{sec::model}. Finally,
we numerically study the effectiveness of our approach in~\Cref{sec::results}.

\section{NONLINEAR PRICING WITH COUPLING COSTS}\label{sec::model}

\subsection{Notation}
For two vectors $x$ and $y$ of $\bbR^d$, we denote by $\left<x,y\right>$ the scalar product and $x\odot y$ the entrywise product. Moreover, for a discrete set $S$, we denote by $|S|$ the cardinality of $S$. 

\subsection{Generalized monopolist problem}
Let us consider a heterogeneous population, where each agent in the population is defined by a $d$-dimensional vector of characteristics $x\in X$. We suppose that $X\subset \bbR^d_{>0}$ is a compact polyhedral domain. An agent of type $x$ will derive a utility $\left<x,\alpha \odot q_k\right> - p_k$ from consuming a good $k$ with quality $q_k\in\bbR^d_{>0}$ and price $p_k\in\bbR_{>0}$. The vector $\alpha\in(\bbR^*)^d$ is an exogeneous data rescaling the quality vector. The agents are distributed according to $\rho$ satysfying $\int_X \rho(x)\dx = 1$.

Let us consider a \emph{monopolist} (principal) who designs a contract menu represented by a pair of functions $x\mapsto (p(x),q(x))\in P \times Q$. For each agent $x$, these functions indicate respectively the price and the quality that the agent is supposed to prefer. Here, $P$ and $Q$ are compact subsets of $\bbR_{>0}$ and $\bbR^d_{>0}$. To ensure that the contract $(p(x),q(x))$ really satisfies agent of type $x$, an additional constraint on the shape of the function, called \emph{incentive-compatibility} condition is required: denoting by $u(x) := \left<x,\alpha \odot q(x)\right> - p(x)$ the utility function for the menu designed by the monopolist,
\begin{equation}~\label{eq::i_compatibility}
u(y)-u(x)\geq \left<y-x,\alpha \odot q(x)\right>, \;\forall x,y \in X \enspace.
\end{equation}
Let $U_x$ be the set of admissible values of $u$ for type $x$:
$$U_x := \{\left<x,\alpha \odot q\right> - p\mid (p,q)\in P\times Q\}\enspace.$$
Each set $U_x$ is compact by compactness of $P$ and $Q$.

\begin{prop}[\cite{Rochet_1987}]\label{prop::rochet_87}
Let $q(\cdot)$ be defined on $ X$, with values in $Q$. There exists a function $p:X\to P$ such that
$u(\cdot)$ satisfies~\eqref{eq::i_compatibility} if and only if
\begin{enumerate}[label=(\roman*)]
\item $u(x)\in U_x,$ for $x\in X$,
\item $u$ is convex on $X$,
\item $ \nabla u(x) = \alpha \odot q(x)$ for a.e. $x\in X$.
\end{enumerate}
\end{prop}

The aim of the monopolist is then to maximize a revenue function, defined as 
\begin{equation}\label{eq::general_obj}
J(u,q):=\int_X L(x,u(x),q(x)) \dx - C\left(\int_X M(x,q(x))\dx\right) \,,
\end{equation}
In~\eqref{eq::general_obj}, the cost function $C$ takes as input data aggregated on the whole domain $ X$. Such coupling cost naturally appears in some applications, for instance in electricity retail market, see~\Cref{sec::app}.
\begin{hypo}\label{hypo::L_and_M}
The integrand $L$ is linear in $u$ and $q$. Moreover, the integrand $M$ is strictly convex in $q$, and $C$ is increasing and strictly convex.
\end{hypo}
In addition to the incentive-compatibility condition, the utility must be greater to a reservation utility:
\begin{equation}\label{eq::reservation_utility}
u(x)\geq R(x) \enspace.
\end{equation}
The problem solved by the monopolist is then
\begin{equation}\label{eq::mf_problem_0}
\max_{u,q} \left\{J(u,q)\;\left\vert 
\begin{split}
&u,q\text{ satisfy }\eqref{eq::i_compatibility},\eqref{eq::reservation_utility}\\
& (u(x),q(x))\in U_x \times Q \text{ for }x \in X
\end{split}
\right.\right\} 
\end{equation}

\begin{theorem}\label{theorem::existence_1}
Under~\Cref{hypo::L_and_M}, Problem~\eqref{eq::mf_problem_0} has a unique optimal solution.
\end{theorem}
The proof of Theorem~\ref{theorem::existence_1} is given in~\Cref{sec::existence}. This result should be compared with~\citep{Carlier_2001}, where the (decomposable) criteria is defined by an integrand that must satisfy coercivity condition, which entails that a minimizing sequence $(u_n)$ must be bounded in the $W^{1,1}$ Sobolev norm.
Here, $J$ is not necessarily coercive. Instead, the compactness argument directly comes with assumptions on $P$ and $Q$.

\subsection{Resolution of the infinite-size case}
As an extension of the monopolist problem, Problem~\eqref{eq::mf_problem_0}  can be solved to optimality through a discretization scheme. In~\citep{Ekeland_2010}, the authors proved the convergence of the discretized problem to the continuous one, which can be extended to nondecomposable cost. Efficient numerical methods have been proposed in~\cite{Carlier_2017} and~\cite{Mirebeau_2016}. Let us define a regular grid $\Sigma$ of $X$. Each of the methods provides a solution $\{(\hat{p}_i,\hat{q}_i)\}_{i\in \Sigma}$, inducing a convex utility function $\hat{u}_{\Sigma}$ that  can be represented as the 
supremum of affine functions, with the notation:
\begin{equation}
\hat{u}_S(x) = \max_{i\in S} \hat{u}_i(x),\;S\subseteq \Sigma\enspace,
\end{equation}
where $\hat{u}_i: x\in\bbR^d\mapsto \left<\hat{q}_{i} ,x\right> - \hat{p}_i$. In the context of max-plus methods~\citep{McEneaney_2007, Akian_2008}, the functions $\hat{u}_i$ are called \emph{basis functions} and can be more general than affine functions, but we focus here on this specific case, as this naturally appears in the model (affine contracts).

\section{PRUNING PROCEDURES}\label{sec::pruning}

\subsection{Pruning method for max-plus basis decomposition}
Let us now suppose that the monopolist has a maximal number of $n$ contracts he can design. Given the discretized infinite-size solution $u_\Sigma$, the question can be recast as the following combinatorial problem:
\begin{equation}\label{eq::criteria_pruning}
\min_{S\subseteq \Sigma} \; \left\{d(\hat{u}_S,\hat{u}_\Sigma) \text{ s.t. } |S|\leq n \right\}\enspace,
\end{equation}
where the function $d(\cdot)$ can be either 
\begin{enumerate}[label=(\roman*)]
	\item the $L_\infty$ norm $d_\infty(u,v)=\left\|u-v\right\|_{L_\infty(X)}$,
	\item the $L_1$ norm $d_1(u,v)=\left\|u-v\right\|_{L_1(X)}$,
	\item and the $J$-based criterion
	 $d_J(u,v)=J(v,\alpha^{\shortminus 1}\odot \nabla v) - J(u,\alpha^{\shortminus 1}\odot\nabla u)\enspace.$
\end{enumerate}
The third case corresponds to the maximization of the function $J$, where $\alpha^{\shortminus 1}\odot q=\nabla u$ thanks to~\Cref{prop::rochet_87}.

\begin{theorem}[\cite{Gaubert_2011}]\label{theorem::gruber}
Let $X\subseteq \bbR^d$ and $v: X \to \bbR$ strongly convex of class $\calC^2$. Then, both $L_1$ and $L_\infty$ approximation errors are $ \Omega\left(\frac{1}{n^{2/d}}\right)$ as $n\to\infty$.
\end{theorem}
Theorem~\ref{theorem::gruber} exhibits an error rate identical to the complexity bound proved in~\cite{Bergemann_2021} in a different setting. 

We define the \emph{importance metric} of basis function $i$ as 
\begin{equation}\label{eq::nu_i}
\nu(S,i) = d(\hat{u}_{S\backslash \{i\}},\hat{u}_S)\enspace.
\end{equation}
This corresponds to an incremental version of the criteria~\eqref{eq::criteria_pruning}.
For the $L_\infty$ and $L_1$ case, if $\nu(S,i) = 0$, then the $i$-th basis function does not contribute to the max-sum.
Otherwise, if $\nu(S,i)>0$, then it expresses the maximal difference between the shape of $\hat{u}_S$ with and without $\hat{u}_i$, depending on the criterion. For the criterion $d_J$, it expresses the loss of revenue for the principal when contract $i$ is removed.

\subsection{Specific case: minimizing $L_\infty$ error}
For a $L_\infty$ approximation error, the importance metric~\eqref{eq::nu_i} can be computed by solving a linear program, see~\cite{Gaubert_2011}:
\begin{equation}\label{eq::pnu_i}
\tag{$P^{S}_i$}
\begin{aligned}
\max_{x\in X,\,\nu} &\quad\nu\\
\text{s.t}&\quad\forall j \in S\backslash\{i\},\quad \hat{u}_i(x) - \hat{u}_j(x) \geq \nu \quad(\lambda_{ij})
\end{aligned}
\end{equation}
In~\eqref{eq::pnu_i}, we denote by $(\lambda_{ij})_j$ the dual variable associated with each constraint. The set of saturated constraints is then characterized by the positive variables $\lambda_{ij}$.

\begin{algorithm}[!ht]
\small
\caption{Pruning for $L_\infty$ importance metric}\label{alg:pruning}
\begin{algorithmic}[1]
\Require $n$\Comment{Desired number of contracts}
\State $S\gets \Sigma$ \Comment{Indices of kept contracts}
\State $I\gets \Sigma$ \Comment{Indices of problems to re-compute}

\For{$t=1:|\Sigma| - n$}
	\For{$i\in I$}
		\State $\nu_{i},\lambda_i \gets$ solution of $(P^{S}_i)$
		\State $J_i\gets \{j\in S\backslash \{i\} \mid \lambda_{ij} > 0\}$
	\EndFor
	\State $r\gets\argmin_{i\in S} \nu_i$ \Comment{Contract to remove}
	\State $S \gets S \backslash \{r\}$
	\State $I \gets \{i\in S \mid r\in J_i \}$
\EndFor
\State\Return $S$
\end{algorithmic}
\end{algorithm}
\Cref{alg:pruning} describes a \emph{greedy descent} procedure: we start from the complete set of contracts $S$, and iteratively remove the less important contract exploiting a fast local update of the importance metric. Compared with~\cite{Gaubert_2011}, we take advantage of the linearity of the basis functions $\hat{u}_i$ to exploit the optimal dual variables $\lambda_{ij}$ in the linear program~\eqref{eq::pnu_i}:
\begin{prop}[Local update]Let $\lambda_{ij}$ be the optimal dual variables in~\eqref{eq::pnu_i} for a contract $i\in S$. Then, the importance metric of $i$ stays \emph{unchanged} when we remove a contract $j\in S$ s.t. $\lambda_{ij}=0$, i.e.,
$\nu(S\backslash\{j\},i) = \nu(S,i)$.
\label{prop::local_update_infty}
\end{prop}

\Cref{prop::local_update_infty} ensures the correctness of~\Cref{alg:pruning}, where  we only re-compute at each iteration the values $\nu_i$ for a very small subset of $\Sigma$. This leads to a huge gain in computation time, see~\Cref{sec::results}.

\subsection{$L_1$ and $J$-based approximation error}
Contrary to the $L_\infty$ case, the computation exploits the geometric structure. Indeed, the representation of the function $\hat{u}_S$ as a maximum of basis functions' $\hat{u}_j,j\in S$ induces a polyhedral complex, in which every function $\hat{u}_i$ determines a polyhedral cell $\mathcal{C}_i$, consisting of the types $x\in X$ such that $\hat{u}_S(x)=\hat{u}_i(x)$. Removing a basis function
$\hat{u}_i$ from the supremum $\hat{u}_S = \sup_{j\in S} \hat{u}_j$ yields a local modification of the latter supremum, concentrated on a neighborhood
of the cell $\mathcal{C}_i$. Hence, we will need to compute at each iteration the neighbors of each contract cell $\mathcal{C}_i$ with $i\in S$.
This idea may be compared with the notion of Delaunay triangulation
associated to a Voronoï diagram~\cite{Fortune_1995}.
During the algorithm, we keep in memory two sets: $J_{i}$ represents the neighboring cells of cell $i$, and $V_{i}$ is the vertex representation of cell $i$. Two routines are used for both the $L_1$ and $J$-based criterion:
\begin{itemize}
\item[$\diamond$] {\sc Vrep}$(S,i)$ returns the V-representation (representation by vertices) of the polyhedral cell $\mathcal{C}_i$ induced by contract $i$ for a given set $S$, taking as input the H-representation (representation by half-spaces) $\{x\in X\mid \hat{u}_i(x)\geq \hat{u}_j(x),\; \forall j\in S\}$ of the cell $i$. This is done using the revised {\em reverse search} algorithm
  implemented in the library \texttt{lrs}, see~\cite{Avis_2000}.
\item[$\diamond$] {\sc updateNeighbors}$((V_S)_{i\in I})$ updates the neighbors of each cell $i\in I$ knowing the vertex representation.
\end{itemize}

\begin{prop}[Local update]The importance metric of a contract $i\in S$ stays \emph{unchanged} when we remove  a contract $j$ which is not in the neighborhood of $i$, i.e.,
$\nu(S\backslash\{j\},i) = \nu(S,i)$ for $j\in S\backslash J_i$.
\label{prop::local_update_1}
\end{prop}

\Cref{prop::local_update_1} ensures the correctness of Algo.~\ref{algo::up1}, where we only re-compute vertex representations for a small subset of contracts (corresponding to the neighboring cells of the lastly removed contract, see line 8 of the algorithm). This local update is illustrated in~\Cref{fig::ex_premerge}. The update of the importance metric in line 11 differs between the $L_1$ and $J$-based cases, and is described in Algos.~\ref{algo::update_nu_1}--\ref{algo::update_nu_J}. 
In Algo.~\ref{algo::update_nu_1}, the integral that appears in the computation of $\nu_i$ can be evaluated analytically using Green's formula,
as it integrates a linear form over a polytope, see~Appendix~\ref{app::green}. In Algo.~\ref{algo::update_nu_J}, $\delta_L$ can be computed in the same way. For $M_0$ and $\delta_M$, this generally involves the integration of the function $x\mapsto M(x,\hat{q}_i)$. In the present application, this function is linear, and so the direct integration is possible, see~\eqref{eq::L}--\eqref{eq::M}.

\begin{algorithm}[!ht]
\small
\caption{Pruning with local update (for $L_1$ and $J$-based)}\label{algo::up1}
\begin{algorithmic}[1]
\Require $n$\Comment{Desired number of contracts}
\For{$i\in \Sigma$}
        \State $V_{i}\gets$ {\sc Vrep}$(\Sigma,i)$ \Comment{Vertex representation}
\EndFor
\State $S\gets \Sigma$ \Comment{Indices of kept contracts}
\State $I\gets \Sigma$ \Comment{Indices of problems to re-compute}

\For{$t=1:|\Sigma| - n$}
	\State $(J_i)_{i\in I}\gets$ {\sc updateNeighbors}$((V_{i})_{i\in I})$
	\For{$i\in I$, $j\in J_i$}
		\State $F_{j,-i}\gets$ {\sc Vrep}$(S\backslash\{i\},j)$ \Comment{Future cells}
	\EndFor
        \State $\nu \gets$ {\sc updateImpMetric}$(I,(V_i)_{i\in S}, (F_{j,\shortminus i})_{j\in J_i, i\in S})$
	
	\State $r\gets\argmin_{i\in S} \nu_{i}$ \Comment{Contract to remove}
        \State $S \gets S \backslash \{r\}$
        \For{$j\in J_r$}
                \State $V_{j}\gets F_{j,-r}$ \Comment{Update vertex representation}
        \EndFor
	\State $I \gets J_r$
\EndFor
\State\Return $S$
\end{algorithmic}
\end{algorithm}
\vspace{-1\baselineskip}
\begin{figure}[!ht]
\centering
\includegraphics[width = 0.8\linewidth,clip=true,trim=1cm .7cm 1cm .5cm]{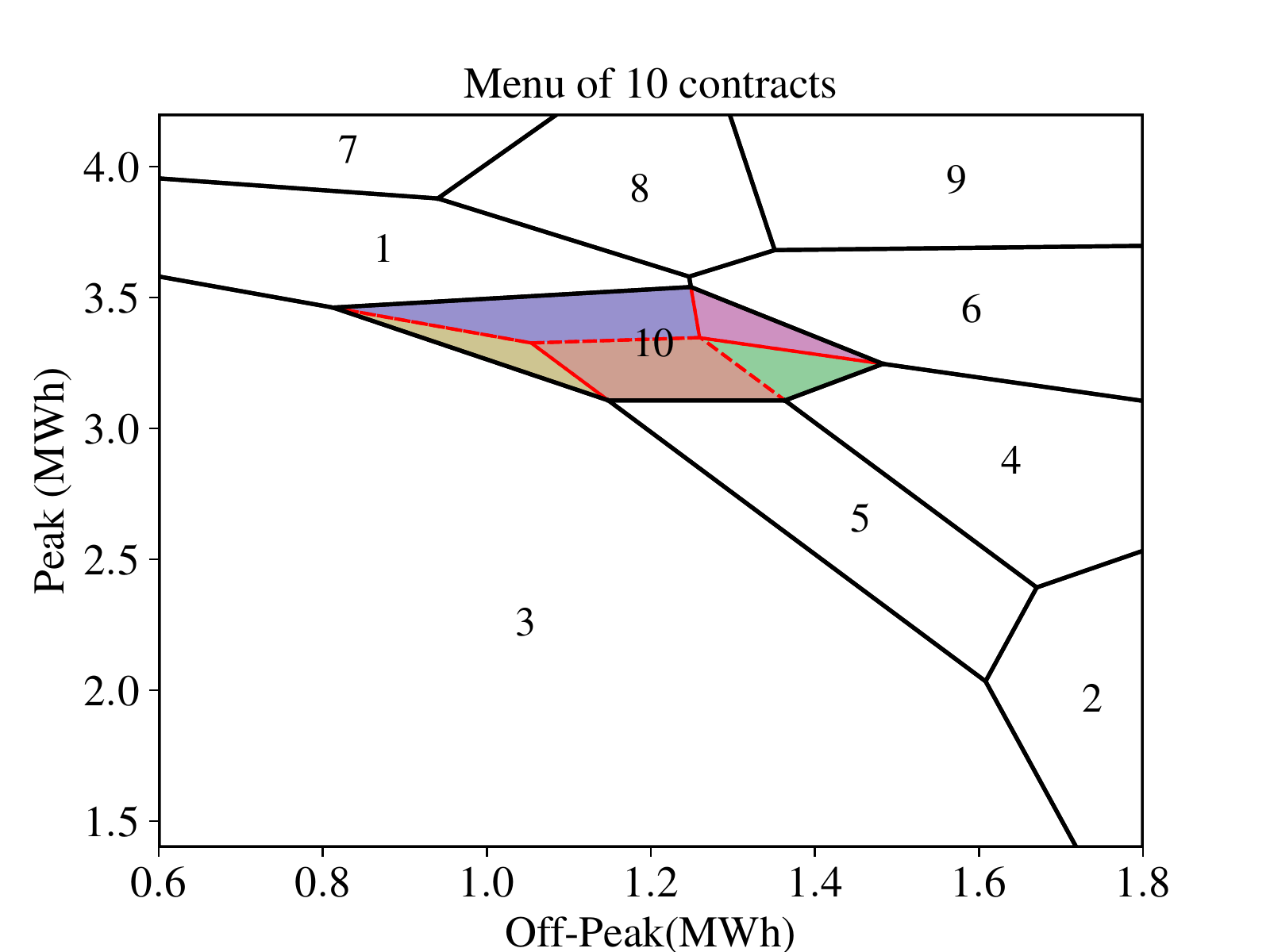}
\caption{Evaluation of contract by dividing into subregions ($d=2$)\\
The green polyhedron corresponds to $F_{4,-10}\cap V_{10}$.}\label{fig::ex_premerge}
\end{figure}

\begin{subalgorithms}
\small
\begin{algorithm}[!ht]
\small
\caption{{\sc updateImpMetric} \; ($L_1$ error)}\label{algo::update_nu_1}
\begin{algorithmic}[1]
\Require $I$, $(V_i)_{i\in S}$, $(F_{j,-i})_{i\in I, j\in J_i}$
\For{$i\in I$}\Comment{Update metric on cells}
        \State $\nu_i \gets \sum_{j\in J_i}\iint_{F_{j,-i}\cap V_i} (\hat{u}_i(x) - \hat{u}_j(x)) \dx$
\EndFor
\State\Return $\nu$
\end{algorithmic}
\end{algorithm}
\begin{algorithm}[!ht]
\small
\caption{{\sc updateImpMetric}\; ($J$-based error)}\label{algo::update_nu_J}
\begin{algorithmic}[1]
\Require $I$, $(V_i)_{i\in S}$, $(F_{j,-i})_{i\in I, j\in J_i}$
\State $M_0\gets \sum_{i\in S}\iint_{V_i} M(x,\hat{q}_i) \dx$
\For{$i\in S$}\Comment{Update metric on cells}
        \State $\delta_L \gets \displaystyle\sum_{j\in J_i}\iint_{F_{j,-i}\cap V_{i}} \hspace{-.8cm}L(x,\hat{u}_i(x),\hat{q}_i) - L(x,\hat{u}_j(x),\hat{q}_j) \dx$
        \State $\delta_M \gets  \sum_{j\in J_i}\iint_{F_{j,-i}\cap V_{i}} M(x,\hat{q}_j) - M(x,\hat{q}_i) \dx$
        \State $\nu_i \gets \delta_L - C(M_0) + C(M_0+\delta_M)$
\EndFor
\end{algorithmic}
\end{algorithm}
\end{subalgorithms}

\begin{prop}[Critical steps]\label{prop::critical_steps}
Let $m$ be the maximum number of neighbors of a polyhedral cell during the execution of the algorithm (for all $t$ and $i$, $|J_i|\leq m$). Then, 
\begin{itemize}
\item[$\diamond$] The number of linear programs $(P^S_i)$ solved
  in Algo.~\ref{alg:pruning} is in $O(m|\Sigma|)$,
\item[$\diamond$]
  The number of computations of a vertex representation of a polyhedral
  cell (calls to {\sc Vrep}$(S,i)$ / reverse search) is in $O(m^2|\Sigma|)$.
\end{itemize}
\end{prop}
By comparison with~\Cref{prop::critical_steps}, a naïve implementation (full recomputation of the importance metric at each step) of the two algorithms would respectively lead to a number of critical steps in $O(|\Sigma|^2)$ and $O(m|\Sigma|^2)$. Each linear program $P^S_i$ can be solved
in polynomial time (by an interior point method).
Reverse search has an incremental running time of $O(|\Sigma|d)$ per
vertex if the input is nondegenerate, see~\cite{Avis_2000}.

\section{APPLICATION TO ELECTRICITY MARKETS}\label{sec::app}

\subsection{Price elasticity}
Let us consider a provider holding several contracts, each of them defined by a fixed price component $p\in\bbR$ (in \texteuro), and $d$ variable price components $z\in\bbR^d$ (in \texteuro/kWh). In France, the contracts
often take into account $d=2$ time periods, with different prices for Peak / Off-peak consumptions.
Moreover, the price coefficients $(p,z)$ of each contract are supposed to belong to a non-empty polytope $P\times Z\subset \bbR^{d+1}$:
\begin{hypo}\label{hypo::order_poly}
Let $p^-,p^+$ be in $\bbR_{>0}$ and $z^-,z^+$ be in $\bbR^d_{>0}$. Then,
$P=[p^-,p^+]$, and the polytope $Z$ is of the following form:
$$Z :=\left\{z^- \leq z\leq z^+\mid z_{i_1} \leq \kappa_{i_1,i_2} z_{i_2}\text{ for } i_1\leq_{\mathcal{P}} i_2\right\}\enspace,$$
where $\mathcal{P}$ is a \emph{partially ordered set} (poset) of $\{1,\hdots,d\}$, and $\leq_{\mathcal{P}}$ the ordering relation, and $\kappa_{i_1,i_2}>0$.
When $\kappa\equiv 1$, $z^-\equiv 0$ and $z^+\equiv 1$, $Z$ is known as an \emph{order polytope}~\citep{Stanley_1986}.
\end{hypo}
\Cref{hypo::order_poly} is natural for the electricity pricing problem: the price can be freely determined within a box (bounds), as long as some inequalities between peak price coefficients and off-peak price coefficients are fulfilled.

We suppose that each agent in the (infinite-size) population is characterized by a \emph{reference} consumption vector $\ce\in X\subset \bbR^d_{> 0}$. 
Here, supposing a continuum of agents is justified since we consider in the application case the population of a whole country.
We suppose that the consumption  is \emph{elastic} to prices, i.e., a consumer can deviate from its reference consumption $\ce$. In addition, we suppose that electricity elasticity can be captured into a utility-based framework, see e.g.~\citep{Samadi_2012} for the properties that the utility must satisfy. Here, we focus on isoelastic utilities: 
\begin{hypo}[Isoelastic utility function]
For a reference consumption $\ce$, the utility of consuming an amount of energy $x\in\bbR^d_{\geq 0}$ is depicted through a \emph{Constant Relative Risk Aversion} (CRRA,\cite{Pindyck_2012,Alasseur_2020}) or isoelastic utility:
\begin{equation}\label{eq::CRRA_function}
\mathcal{U}_{\ce}:x\in\bbR^d_{\geq 0}\mapsto \frac{1}{\eta}\sum_{i=1}^d \beta_{\ce i} (x_i)^\eta,\;\eta\in (\shortminus\infty,0)\cup(0,1]\enspace.
\end{equation}
The coefficient $\eta$ is called the \emph{risk aversion} coefficient. 
\end{hypo}
In this context, this elasticity measure depicts the easiness of a customer to adopt another energy source to fulfill his needs. 
In~\citep{Alasseur_2020}, the authors model the electric elasticity by this kind of utility function, and separate the case $\eta < 0$ and $\eta \in (0,1]$.
  The first regime ($\eta<0$) will model a household consumption: the satisfaction coming from consuming energy saturates to a maximum utility,
and a zero consumption is prohibited. In contrast, the second regime ($\eta\in(0,1]$) will represent the high flexibility of the industrial sector, which can adapt more easily its consumption according to price. 
We refer to~\citep{Niromandfam_2020_bis} and references therein for empirical studies on the intensity of the elasticity coefficient $\eta$. 

For a contract defined by price coefficients $(p,z)\in\bbR\times\bbR^d$, a consumer $\ce$ will optimize his consumption in order to maximize the \emph{welfare function}, obtained by subtracting the electricity cost to~\eqref{eq::CRRA_function}:
\begin{equation}
\label{eq::welfare_max}
\calU^*_{\ce}: (p,z)\in\bbR\times\bbR^d\mapsto \max_{x\in{\bbR_{\geq 0}}^d}\left\{\calU_{\ce}(x) -\left<x,z \right>\right\} - p \enspace.
\end{equation}
We denote by $\calU^*_{\ce}$ the welfare function as the maximization term in~\eqref{eq::welfare_max} corresponds to a Fenchel-Legendre transform up to a change of sign. As a consequence, $\calU^*_{\ce}$ is convex and nonincreasing. We now make the following assumption to fix the value of $\beta$:

\begin{hypo}\label{hypo::alpha}
The reference consumption $\ce\in\bbR^d$ is obtained for reference prices $\check{p}\in\bbR$ and $\check{z}\in\bbR^d$. 
\end{hypo}

Under~\Cref{hypo::alpha}, the optimal consumption of customer $\ce$ on period $i\in \{1,\hdots,d\}$, denoted $\calE_{\ce i}$, is given by
\begin{equation}\label{eq::optimal_energy}
\calE_{\ce i}(z) = \ce_i\left(z_i / \cz_i\right)^{\frac{-1}{1-\eta}}\; \geq 0\enspace,
\end{equation}
and the welfare function is given by 
\begin{equation}\label{eq::Ustar}
\calU^*_{\ce}(p,z) = \left(\tfrac{1}{\eta} - 1\right)\sum_{i=1}^d\ce_h\cz_h\left(z_h / \cz_h\right)^{\frac{-\eta}{1-\eta}} - p\enspace.
\end{equation}
\Cref{eq::optimal_energy,eq::Ustar} are obtained from the first order optimality condition (zero derivative) for~\eqref{eq::welfare_max} ($\beta_{\ce i} = \cz_h\left(\ce_h\right)^{1-\eta}$).

\subsection{Infinite-size menu of offers}
In this section, we relax the assumption of a finite number of contracts, by supposing that the provider is able to define \emph{as many offers as} consumers. Therefore, the infinite-size menu of offers can be represented by two functions $p: X\to\bbR$ and $z: X\to\bbR^d$, representing respectively the fixed price component and the variable price components.
Let us define the (weighted) invoice of a consumer as
\begin{equation}
\calL_{\ce}:(p,z)\in \bbR\times \bbR^d \mapsto (p+ \left<\calE_{\ce}(z),z\right>)\rho(\ce) \enspace,
\end{equation}
where $\rho(\check{x})\geq 0 $ represents the density of customers with reference consumption $\check{x}$. The provider's revenue maximization problem is then
\begin{subequations}\label{eq::mf_problem}
\begin{align}
\max_{p,z} & \;\calJ^1(p,z) - \calJ^2(z)\\
\text{s.t.}&\;\calU^*_{x}(p(x),z(x))\geq \calU^*_{x}(p(y),z(y)),\,\forall x,y\in X\label{eq::compatibility}\\
&\;\calU^*_{x}(p(x),z(x)) \geq R(x),\,\forall x\in X\label{eq::reservation}\\
&\;p(x)\in P,\;z(x)\in Z
\end{align}
\end{subequations}
where $\calJ^1(p,z)= \int_{ X} \calL_{x}(p(x),z(x)) \dx$ and  $\calJ^2(z)= C\left(\int_{X} \sum_{i=1}^d \calE_{xi}(z(x)) \rho(x) \dx \right)$.

\Cref{eq::compatibility,eq::reservation} are respectively the \emph{incentive-compatibility condition} and \emph{participation constraint}. Taking $C$ as a strictly convex increasing function of the global consumption is often considered in the literature. 
In particular, this cost function is often modeled as a piecewise linear function, see e.g.~\cite{Alekseeva_2019}, or as a quadratic function, see e.g.~\cite{Ackooij_2018}. In fact, the marginal cost to supply electricity is not constant and increases with the consumption. 
The convexity of the reservation utility is also a classical assumption, as this reservation utility should be a supremum over the utilities of alternative offers (each of them being linear function of the reference consumption).

Let us make the following change of variables:
$$q_{i}:= (z_{i}/\cz_{i})^{\frac{-\eta}{1-\eta}}\enspace.$$ 
Then, the consumption on period $i\in \{1,\hdots,d\}$ is a convex function of $q_i$, expressed as
$\mathfrak{E}_{\ce i}(q_{i}) = \ce_i[q_{i}]^{\frac{1}{\eta}}\enspace,$
and both the utility and the weighted invoice now read as linear functions of $p$ and $q$: defining $\alpha = (\eta^{\shortminus 1} - 1) \cz$,
\begin{equation}\label{eq::L}
\begin{aligned}
u(x) &:= \left<x , \alpha\odot q(x)\right> - p(x)\enspace,\\
L(x,u(x),q(x)) &:= \left(\tfrac{1}{\eta}\left<x, \cz \odot q(x)\right> - u(x)\right)\rho(x)\enspace,
\end{aligned}
\end{equation}

\begin{theorem}\label{theorem::mf_problem_2}
Under~\Cref{hypo::order_poly}, the provider's revenue maximization problem~\eqref{eq::mf_problem} is equivalent to a monopolist problem of the form~\eqref{eq::mf_problem_0}
with 
\begin{equation}\label{eq::M}
M(x,q(x)) := \rho(x)\sum_{i=1}^d x_i [q_i(x)]^{\tfrac{1}{\eta}}
\end{equation}
and, if $\eta <0$,
\begin{equation*}
Q = \left\{q\in\bbR^d\left\vert
\begin{split}
&\left(z^-/\cz\right)^{\frac{-\eta}{1-\eta}}\leq q \leq \left(z^+/\cz\right)^{\frac{-\eta}{1-\eta}} \\
& q_{i_1} \hspace{-.01\linewidth} \leq \left(\kappa_{i_1,i_2}\tfrac{\cz_{i_2}}{\cz_{i_1}}\right)^{\frac{-\eta}{1-\eta}} q_{i_2}\text{ for } i_1\leq_{\mathcal{P}} i_2
\end{split}\right.\right\},
\end{equation*}
otherwise, 
\begin{equation*}
Q = \left\{q\in\bbR^d\left\vert
\begin{split}
&\left(z^+/\cz\right)^{\frac{-\eta}{1-\eta}}\leq q \leq \left(z^-/\cz\right)^{\frac{-\eta}{1-\eta}} \\
& q_{i_1} \geq \left(\kappa_{i_1,i_2}\tfrac{\cz_{i_2}}{\cz_{i_1}}\right)^{\frac{-\eta}{1-\eta}} q_{i_2}\text{ for } i_1\leq_{\mathcal{P}} i_2
\end{split}\right.\right\}
\end{equation*}
\end{theorem}
\begin{proof}
Owing to assumption on the set $Q$ and the strict monotonicity of $z\mapsto z^{\frac{-\eta}{1-\eta}}$ (increasing for $\eta <0$ and decreasing for $\eta>0$), one can explicitly derive the form of $Q$. The rest of the formulation is immediate.
\end{proof}
\section{NUMERICAL RESULTS}\label{sec::results}

\subsection{Instance}
The numerical results were obtained on a laptop i7-
1065G7 CPU@1.30GHz. We provide in~\Cref{table::instance} the values of the parameters used in the application. In particular, we consider reference prices $(\hat{p},\hat{z})$ corresponding to French regulated prices, and reference consumption spread around the mean French consumption per household ($\calE_{\text{moy}} = 4$MWh). The cost function is taken as a quadratic function, scaled so that the marginal cost $C'(\calE_{\text{moy}}) = 0.08$\texteuro/kWh. In comparison, the production cost is estimated in France around 0.05\texteuro/kWh for nuclear plants\footnote{CRE (2022), \emph{Délibération n° 2022-45}} and up to 0.09\texteuro/kWh for wind energy\footnote{ADEME (2016), \emph{Coûts des énergies renouvelables en France}}.

\begin{figure*}[t]
\centering
\begin{subfigure}{0.32\linewidth}
\includegraphics[width=\linewidth,clip=true,trim=0.3cm 0cm 1cm 1.3cm]{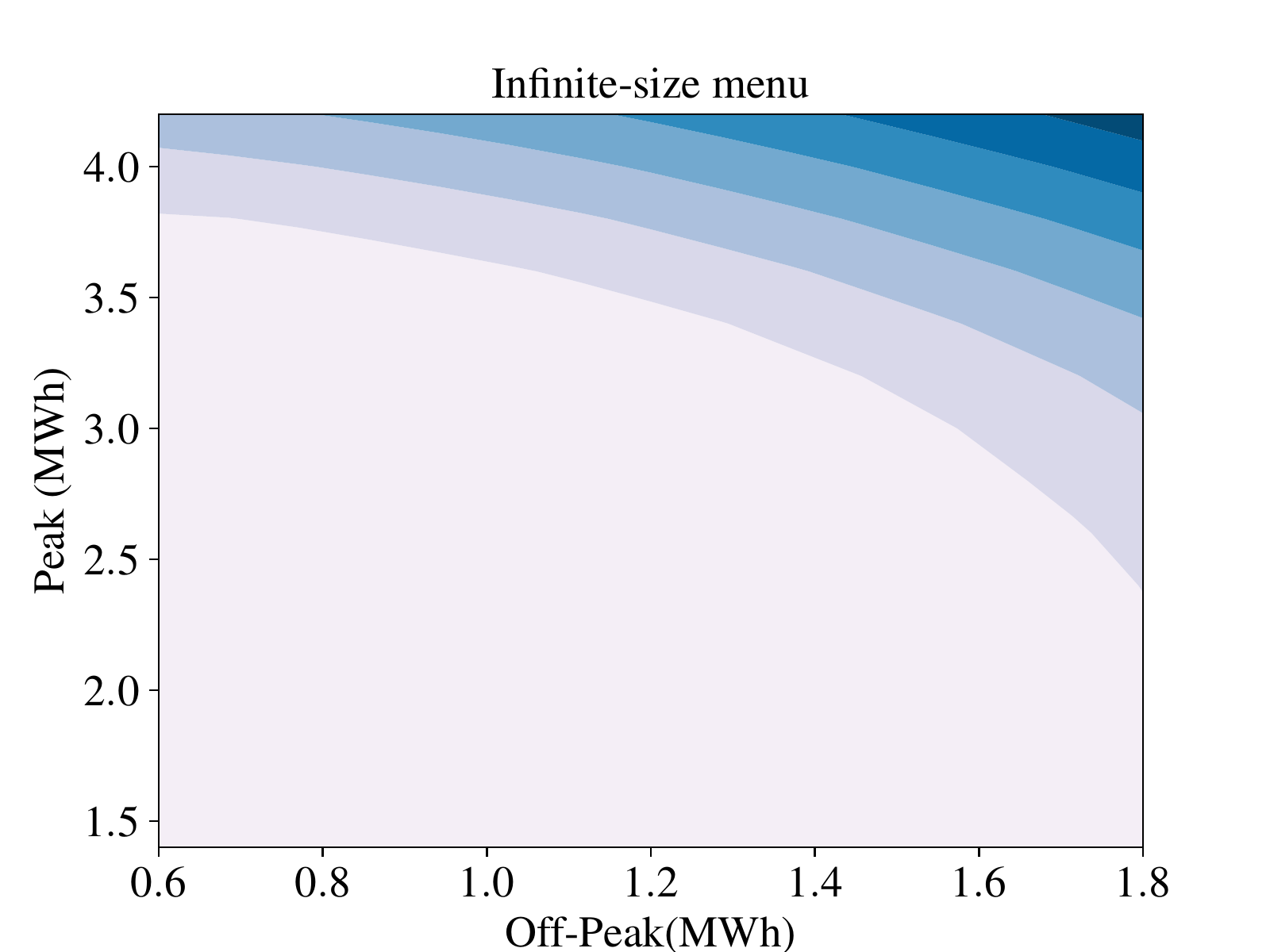}
\caption{Infinite-size menu}\label{fig::quantization_a}
\end{subfigure}
\begin{subfigure}{0.32\linewidth}
\includegraphics[width=\linewidth,clip=true,trim=0.3cm 0cm 1cm 1.3cm]{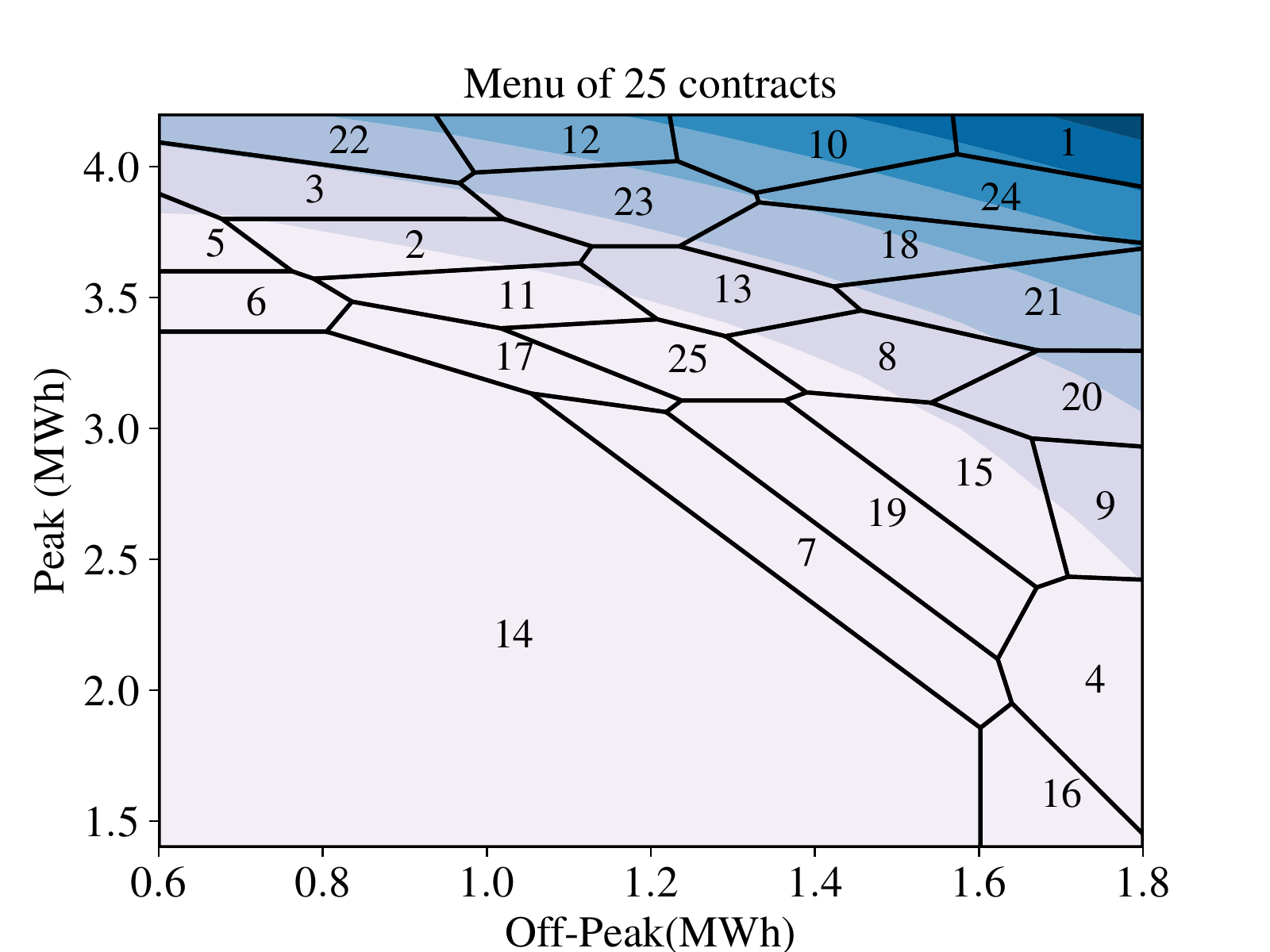}
\caption{Menu of 25 contracts}\label{fig::quantization_b}
\end{subfigure}
\begin{subfigure}{0.32\linewidth}
\includegraphics[width=\linewidth,clip=true,trim=0.3cm 0cm 1cm 1.3cm]{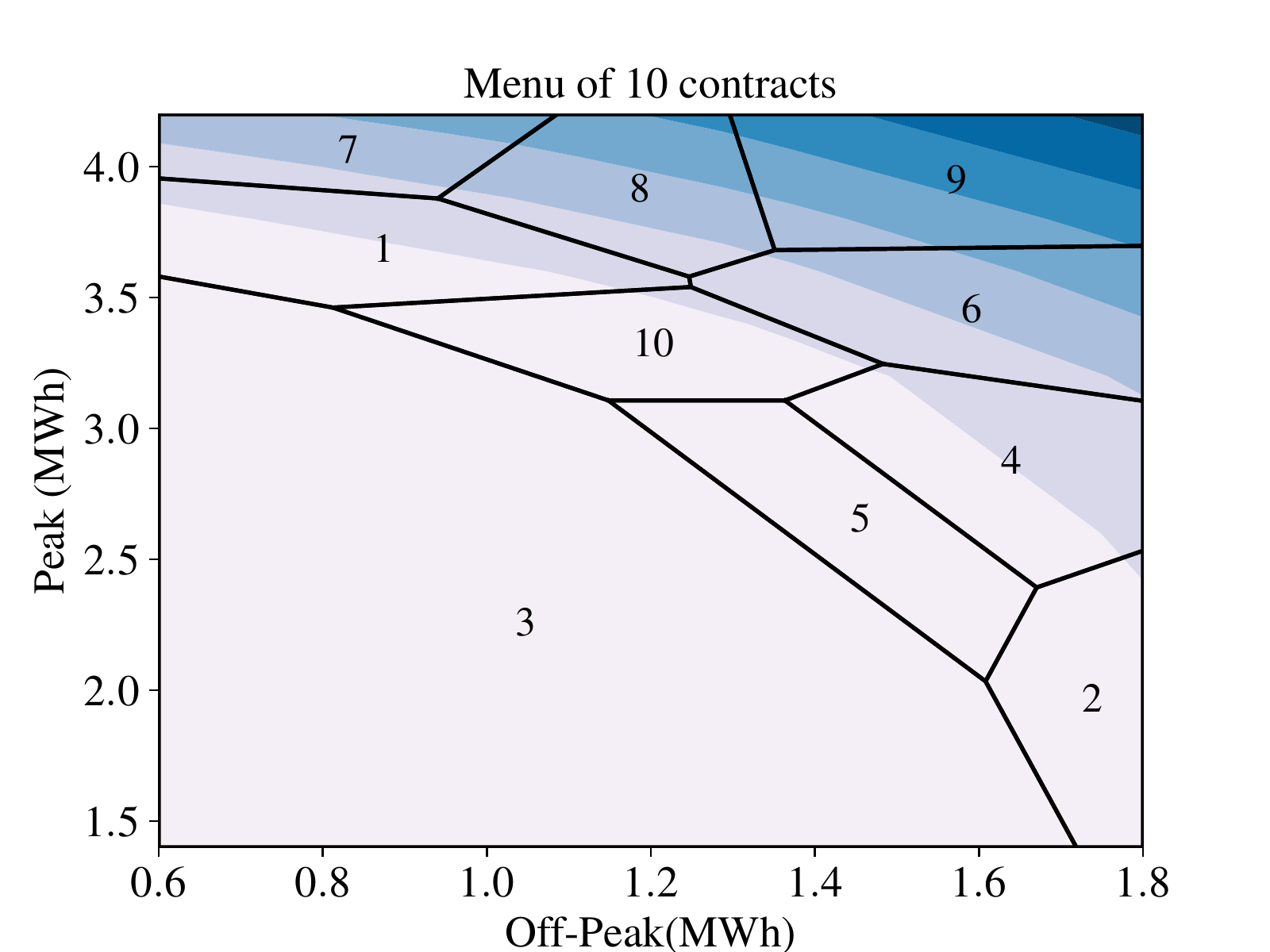}
\caption{Menu of 10 contracts}\label{fig::quantization_c}
\end{subfigure}
\caption{$L_1$-norm pruning for the electricity market case.\\
The normalized utility $u - R$ is depicted with colormap (light gray corresponds to the zero value and blue to high value).}\label{fig::quantization}
\end{figure*}

\begin{table}[!ht]
\centering
\begin{tabular}{r|l}
$\eta$ & -0.1\\
$\check{p}$ & 140\texteuro\\
$\cz$ & (0,174,019)\texteuro/kWh\\
$C(\cdot)$ & $0.01(\cdot)^2$\\
$(p^-,p^+)$ & $(0,500)$\texteuro\\
$(q_1^-,q_1^+)$ & $(0.05,0.5)$\texteuro/kWh\\
$(q_2^-,q_2^+)$ & $(0.05,0.5)$\texteuro/kWh\\
$\rho$ & Uniform$([0.6,1.8]\times [1.4,4.2])$\\
$R(\cdot)$ & linear function (one regulated contract)
\end{tabular}
\caption{Instance used in the numerical results.}\label{table::instance}
\end{table}
We display in~\Cref{fig::quantization} the infinite-size menu and the quantized solution for two different sizes of menu (25 contracts and 10 contracts). In each cell $\calC_i$, the contract $i$ brings to customers of reference consumption $x\in\calC_i$ the maximal utility given the quantized menu, i.e., $\hat{u}_S(x) = \hat{u}_i(x)$ for $x\in\calC_i$. 
We observe that there is a region/cell (light gray region) where the monopolist reproduces the alternative option (of utility $R$). On the other side, for high consumption (peak or off-peak), the monopolist manages to design contracts that provide strictly higher utility than the regulated offer, and at the same time, procure to the monopolist a higher revenue.

\subsection{Comparison of pruning objectives}
In the upper graph of~\Cref{fig::errors}, the three pruning objectives studied in the paper ($L_\infty$, $L_1$ and $J$-based) are compared with the 1-step approach of~\cite{McEneaney_2008,Gaubert_2011}. The approach consists in sorting the importance metrics for all $i\in\Sigma$, and directly taking the $n$ contracts with highest importance metric (here we consider the $J$-based importance metric). We display the relative objective loss, defined as 1 - $J_t/J_{\text{ref}}$, where $J_t$ is the objective for a menu of size $t$ and $J_{\text{ref}}$ the objective obtained with the infinite-size menu. Note that removing a contract can induce a violation of the full-participation constraint ($u\geq R$). Therefore, in order to recover a feasible solution at each iteration, we lift up the solution with the simple rule $u\gets u + \max\{\max_{x\in X} \{R(x)-u(x)\},0\}$.

On this example, the pruning procedure of Algo.~\ref{algo::up1} (greedy descent) leads to a significant loss reduction, whatever the criterion, compared with the 1-step approach. As expected, we observe that the $J$-based pruning has the smallest relative loss in the objective, as we minimize the error at each iteration of the process. In contrast, the $L_\infty$-norm does not capture sufficiently well the behavior of the objective function $J$, and has larger objective loss, even for a large number of contracts. 

We also depicted the cumulated time along the iterations in the lower graph of~\Cref{fig::errors} (we do not display the time for the 1-step procedure, as it is very fast, in less than 0.5s). For comparison, we add the cumulative time of a ``naïve'' $J$-based pruning, recomputing at each iteration the importance metric of each cell (global update). On this example, we observe that the computational time is already reduced by a factor almost 3 (this factor would be greater in higher dimension, as the neighborhood would be larger). As expected, the $L_\infty$ criterion is the fastest, owing to the fast local update rule exploiting the sparsity of optimal Lagrange multipliers (\Cref{alg:pruning}), 
and the $J$-based and $L_1$-norm criteria have similar computational time, as they use the same algorithmic architecture, see Algo.~\ref{algo::up1}. In terms of
loss minimization, the $J$-based pruning shows a loss of revenue reduced
by a factor of around $2$ by comparison with other methods.  This approach
allows us to determine the minimum number of contracts given an
admissible revenue loss: e.g., \Cref{fig::errors} shows that,
with a $J$-based quantization, a menu of $10$ contracts suffices to limit
the revenue loss to $4$\%.
\begin{figure}[!ht]
\centering
\includegraphics[width=0.95\linewidth]{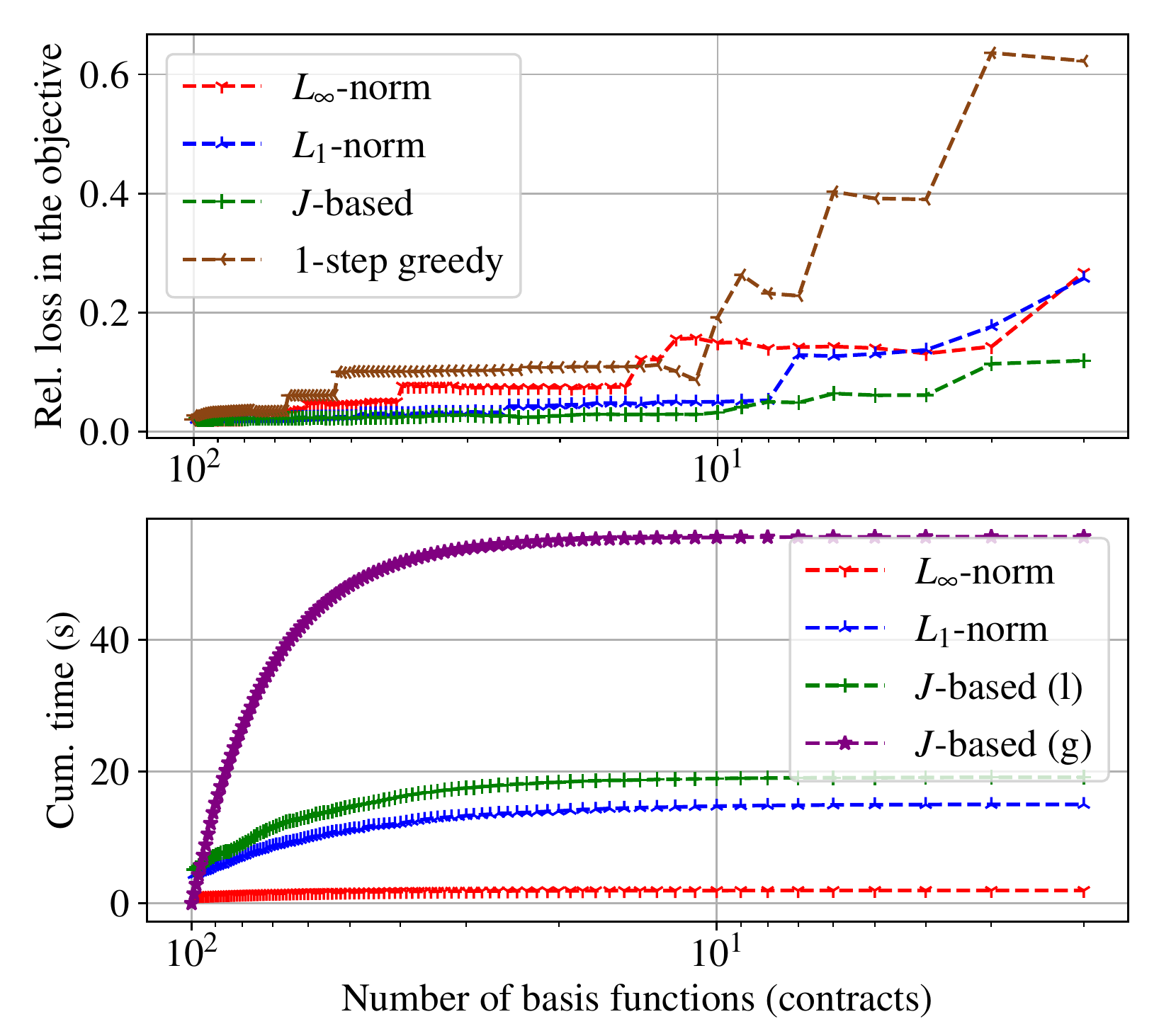}
\caption{Comparison of error bounds for the three types of pruning objective. 
The upper graph shows the loss of optimality induced by a reduced number of contracts.
The lower graph depicts cumulative time along the iterations. (g) stands for global update while (l) stands for local update.}\label{fig::errors}
\end{figure}

\section{CONCLUSION}
We have addressed a nonlinear pricing problem incorporating coupling costs. This  arises in electricity markets, where supply costs depend on the global consumption. We have developed a quantization procedure, allowing to maximize
the revenue of a provider, given a cardinality constraint on the set
of contracts. This relies on refined pruning procedures, inspired
by the max-plus basis methods in numerical optimal control. In particular,
we exploited the local nature of the pruning process, in order
to reduce the computational time. Thus, this leads to a new
class of applications for methods originally developed in
optimal control, and this also improves the complexity of a key
ingredient of these methods.

A strong parallel with vector quantization can be made, see e.g.~\cite{Pages_2015}. In this context, a different quantization problem is addressed
by Lloyd's procedures, {\it ibid.}. Whether these ideas
can be adapted to the quantization of the maximum of affine functions with revenue criterion is left for further work.


\printbibliography[heading=bibintoc]

\section{APPENDIX}\label{app::proof}

\subsection{Proof of~\Cref{theorem::existence_1}}\label{sec::existence}
Using~\Cref{prop::rochet_87}, for any solution, $\nabla u(x)=\alpha \odot q(x)$ for a.e. $x\in X$. We then directly study the existence and uniqueness in $(u,\nabla u)$.

\textbf{Existence.}
Let $H^1( X)$ be the Sobolev space associated with $X$. We define 
$$\calK =\left\{u \left\vert 
\begin{split}
&u \text{ convex and } u\geq R\\
&u(x)\in U_x,\; \alpha^{\shortminus 1}\odot \nabla u(x) \in Q\;,\;\forall x \in X
\end{split}
\right.\right\}\enspace.$$ 
The set $\calK$ is a closed, convex and bounded subset of $H^1( X)$ (it is bounded since $ X$ is bounded and $\|u\|_{L_\infty}$ and $\| \nabla u\|_{L_\infty}$ are bounded too; it is convex since $R$ is convex).

Besides, $J$ is concave (\Cref{hypo::L_and_M}). Moreover, as $Q\subset \bbR_{>0}^d$, there exist $a,b>0$ such that for any $x\in X$, $a\leq \|\nabla u(x)\|\leq b$. Therefore, there exists $c\in\bbR_+$ such that $|J(u,\nabla u)|\leq c$, and as a consequence, $J$ is continuous on $\calK$, see~\cite[Chapter 1, Proposition 2.5]{Ekeland_1999}.

Using the fact that $H^1( X)$ is reflexive and~\cite[Chapter 2, Proposition 1.2]{Ekeland_1999}, Problem~\eqref{eq::mf_problem_0} admits at least one solution.

\textbf{Uniqueness.} (Same arguments as in~\cite{Rochet_1998})
Let now consider two distinct solutions $u_1$ and $u_2$. Then, 
if $\nabla u_1 \neq \nabla u_2$ on a measurable subset, any function $t u_1 + (1-t) u_2$ is valid and gives a strictly better solution than $u_1$ and $u_2$ (due to strict convexity of the cost function $u\mapsto C(\int_X M(x,\nabla u(x))\dx)$ and linearity of $L$). Therefore, $u_1 - u_2$ is a constant function. By linearity of $L$, the objective value obtained with $u_1$ and $u_2$ differ by the same constant. This contradicts the optimality of the two solutions $u_1$ and $u_2$.

\subsection{Fast metric updates using Green's formula}\label{app::green}
The next proposition allows us to implement efficiently the local updates of the importance metric performed in~\ref{algo::update_nu_1}--\ref{algo::update_nu_J}.
\begin{prop}
Let $P$ a 2D-polytope describes by its vertices $(x_i,y_i)\in\bbR^2$ (counter-clockwise ordered). Then for any $a,b,c \in \bbR$,
\vspace{-1\baselineskip}
\begin{equation*}
\iint_P (ax+by+c)dxdy = \sum_{i = 1}^N \left[
\begin{split}
&\oint_{y_i}^{y_{i+1}}b(q_i+\tfrac{1}{\tau_i}y)ydy \\
&\shortminus \oint_{x_i}^{x_{i+1}}\hspace{-.5cm}(ax+c)(p_i+\tau_i x)dx
\end{split}
\right],
\end{equation*}
with $\tau_i = \frac{y_{i+1} - y_i}{x_{i+1} - x_i}$, $p_i:= y_i-\tau_i x_i$ and $q_i := x_i - \frac{1}{\tau}y_i$.
\end{prop}
\begin{proof}
The application of the Green formula gives :
\begin{equation*}
\iint_P (ax+by+c)dxdy = \oint_{\calC_P} (bxy) dy - (ax+c)ydx\enspace,
\end{equation*}
where $\calC_P$ is the contour of the polytope $P$.
We then decompose on each edges, and use the change of variable $x = q + y/\tau$ in the first integral and $y = x+\tau x$ in the second one.
\end{proof}

\end{document}